\documentclass{article}

\usepackage{graphicx}
\usepackage{cite}
\usepackage{authblk}
\usepackage{amsthm}
\usepackage{amsmath}
\usepackage{amssymb}
\usepackage{amsfonts}
\usepackage{amscd}
\usepackage{indentfirst}
\usepackage{comment}
\usepackage[left=3cm, right=3cm, lines=45, top=1.0in, bottom=1.0in]{geometry}
\usepackage[colorlinks=true]{hyperref}

\hypersetup{
colorlinks=true,
linkcolor=black,
citecolor=black,
urlcolor=black
}

\date{}
\title{A Local Lewy Theorem for $p$-Harmonic Function with Non-zero Gradient in  $R^3$}
\author{Jiahuan Li, Yilu Liu, Xi-Nan Ma}

\newcommand{\keywords}[1]{\par\quad\textbf{Keywords:} #1}

\numberwithin{equation}{section}

\newtheorem{proposition}{Proposition}[section]
\newtheorem{theorem}[proposition]{Theorem}
\newtheorem{lemma}[proposition]{Lemma}

\DeclareMathOperator{\rank}{rank}
\DeclareMathOperator{\tr}{tr}
\DeclareMathOperator{\sgn}{sgn}
\DeclareMathOperator{\diver}{div}
\DeclareMathOperator{\diag}{diag}

\allowdisplaybreaks[4]

\begin{document}

\maketitle

\begin{abstract}
We establish a local Lewy-type theorem for 
\(p\)-harmonic function  with non-zero gradient  in dimension three space $R^3$. Let \(1<p<\infty\), and let
\(u\in W^{1,p}_{\mathrm{loc}}(\Omega)\) be a weak \(p\)-harmonic function in
a domain \(\Omega\subset\mathbb R^3\), assume it satisfies \(|Du|>0\), we prove that a locally homeomorphic
gradient map \(Du\) must have non-vanishing Hessian determinant. Hence
\(Du\) is a local diffeomorphism. 
\end{abstract}

\keywords{$p$-Laplace equation; gradient homeomorphism; Lewy theorem; partial Legendre transform.}

\section{Introduction}

The study of gradient mappings associated with elliptic equations lies at
the intersection of nonlinear potential theory, geometric analysis, and
topological mapping theory. A basic question is whether the topological
non-degeneracy of a gradient map forces its differential to be
non-degenerate. For a twice differentiable potential $u$, this asks whether
a homeomorphic gradient $Du$ must satisfy $\det D^2u\neq0$. Even for linear
elliptic equations, the answer depends strongly on the dimension and on the
structure of the equation.

In dimension three, Lewy proved that if $u$ is harmonic and $Du$ is a
homeomorphism, then $Du$ is in fact a diffeomorphism \cite{Lewy1968}.
Equivalently, the Hessian determinant of a harmonic function cannot vanish
at a point while its gradient remains locally one-to-one. Lewy's argument is
driven by a maximum principle for the Hessian determinant and the minimal surface theory. Gleason and Wolff
subsequently developed this point of view in higher dimensions and isolated
a robust eigenvalue form of the three-dimensional maximum principle
\cite{GleasonWolff1991}. These results show that a condition expressed in
terms of topology can, for harmonic gradients, enforce a sharp analytic
non-degeneracy conclusion.

The nonlinear analogue considered here is the scalar $p$-Laplace equation.
Its regularity theory was developed through foundational work on degenerate
and quasilinear elliptic equations, including the contributions of Ural'ceva, Uhlenbeck,
Lewis, and Tolksdorf \cite{1968LOMI,Uhlenbeck1977,Lewis1977,Tolksdorf1984}. At points
where the gradient vanishes, the equation is degenerate or singular when
$p\neq2$. Away from the critical set, however, it becomes a uniformly
elliptic equation with analytic coefficients; standard elliptic regularity
then makes the Hessian and its determinant available for a pointwise study.
This non-critical regime is therefore the natural setting in which to seek a
Lewy-type theorem for $p$-harmonic function.

The purpose of this paper is to establish the corresponding local mechanism
for the scalar $p$-Laplace equation in dimension three under the
non-critical hypothesis. Precisely, we study weak solutions
$u\in W^{1,p}_{\mathrm{loc}}(\Omega)$ of
\begin{equation}\label{eq:plaplace}
        \diver\left(|Du|^{p-2}Du\right)=0,
        \qquad 1<p<\infty,
\end{equation}
under the standing hypothesis, imposed on the $C^{1,\alpha}$ representative,
\begin{equation}\label{eq:noncritical}
        |Du|>0.
\end{equation}
By the standard non-critical regularity theorem, $u$ is actually real analytic. Hence
\eqref{eq:plaplace} is classical and, in the non-critical region, equivalent
to the uniformly elliptic analytic non-divergence equation
\begin{equation}\label{eq:nondiv}
        a_{ij}(Du)u_{ij}=0,
        \qquad
        a_{ij}(q)=\delta_{ij}+(p-2)\frac{q_iq_j}{|q|^2}.
\end{equation}

Our main result is the following.
\begin{theorem}\label{thm:main}
Let $\Omega\subset\mathbb R^3$ be a domain, let $1<p<\infty$, and let
$u\in W^{1,p}_{\mathrm{loc}}(\Omega)$ be a weak solution of
\eqref{eq:plaplace}. Assume that  $u$
satisfies
\[
        |Du|>0
\]
in $\Omega$.
If
\[
        Du:\Omega\longrightarrow Du(\Omega)
\]
is a local homeomorphism, then
\[
        \det D^2u(x)\neq 0
        \qquad x\in\Omega.
\]
In particular, $Du$ is a local diffeomorphism. If $Du$ is globally
one-to-one onto its image, then $Du$ is a diffeomorphism onto its image.
\end{theorem}

The proof of the above Theorem~\ref{thm:main} divide theree cases according to the rank of the Hessian matrxi $\{u_{ij}\}$. When the rank is zero,  we use the reduction rsult by Gordan--Noether theorem \cite{WatanabeDeBondt2017}. When the rank of  $\{u_{ij}\}$ is one, it is easy case.  When the rank of  $\{u_{ij}\}$ is two, then we use the partial Legendre transform, which had been developed in Guan\cite{1997advmath}, Guan-Sawyer\cite{2009TAMS} and Rios-Sawyer-Wheeden\cite{AdvMath05}, and we use the maximum principle to get the proof.

The paper is organized as follows. In Section~2 we collect the analytic,
topological, and algebraic facts used throughout the proof, including the
Lewy--Gleason--Wolff maximum principle and a real form of the ternary
Gordan--Noether theorem. We shall use these facts to prove the Theorem~\ref{thm:main} in rank zero case in section 5. Section~3 excludes the rank-one case. Section~4
handles the rank-two case by means of a partial Legendre transform and a
sheet-order obstruction. Section~5 treats the rank-zero case through a
blow-up analysis and the structure of homogeneous harmonic polynomials.
Finally, Section~6 combines the three alternatives and proves
Theorem~\ref{thm:main}.

\textbf{Acknowledgement } The  third author thanks Professor Fanghua Lin for bringing this question to his attention in some years ago. The authors are supported by National Key R\&D Program of China 2025YFA1017603. 

\section{Preliminaries and Some Basic Facts}

We first isolate some basic facts used in the proof.

\begin{theorem}\label{thm:LGW}
Let $Q$ be harmonic in $\Omega$. Suppose that, at every
point, $D^2Q$ has at most one negative eigenvalue. If
\[
        \det D^2Q(x_0)=0
\]
at some point $x_0$, then
\[
        \det D^2Q\equiv 0
\]
in the connected component containing $x_0$.
\end{theorem}

For dimension three this is precisely the maximum principle behind Lewy's
theorem. In the form stated above it is Theorem 1 of Gleason--Wolff
\cite{GleasonWolff1991}. In the application below, the hypothesis is checked
from the sign of the Hessian determinant. Indeed, if $M$ is a real symmetric
$3\times3$ matrix with
\[
        \tr M=0,
\]
then
\[
        \det M\leq0
        \quad\Longrightarrow\quad
        M\text{ has at most one negative eigenvalue},
\]
whereas
\[
        \det M\geq0
        \quad\Longrightarrow\quad
        -M\text{ has at most one negative eigenvalue}.
\]
This follows immediately from the three real eigenvalues
\[
        \lambda_1+\lambda_2+\lambda_3=0.
\]
Thus, in dimension three, a harmonic function whose Hessian determinant has
a fixed weak sign satisfies the Gleason--Wolff eigenvalue hypothesis after
possibly replacing the function by its negative.

\begin{lemma}\label{lem:sign}
Let $F\in C^\omega(\Omega,\mathbb R^3)$ be a local homeomorphism. Then $\det DF$ has a fixed weak sign in $U$.
Equivalently, either
\[
        \det DF\geq 0
\]
throughout $U$, or
\[
        \det DF\leq 0
\]
throughout $U$.
\end{lemma}

\begin{proof}
A local homeomorphism between oriented connected $3$-manifolds has a fixed
orientation character. Equivalently, the local topological degree
\[
        \deg(F,B_x,F(x))
\]
is independent of $x$, after the ball $B_x$ is chosen sufficiently small.
This degree is either $1$ or $-1$; see, for example,
\cite[Chapter 2]{Hatcher2002}. At a regular point it agrees with
$\sgn\det DF$. Hence all regular points have the same Jacobian sign.

It remains only to know that regular points are dense. If $\det DF$ vanished
on a non-empty open ball $B$, then the area formula for the locally Lipschitz
map $F$ would give
\[
        |F(B)|\leq \int_B |\det DF|\,dx=0;
\]
see \cite[Section 3.4]{EvansGariepy1992}. This contradicts the fact that a
local homeomorphism is an open map. Thus the analytic function $\det DF$
does not vanish identically on any open ball, and its non-zero set is dense.
Continuity gives the asserted weak sign.
\end{proof}

\begin{theorem}\label{thm:Hesse}
Let $Q$ be a real homogeneous polynomial in three variables, of degree at
least three. If
\[
        \det D^2Q\equiv 0,
\]
then, after an invertible real linear change of variables,
\[
        Q(X)=Q_0(X_1,X_2)
\]
for a homogeneous polynomial $Q_0$ in two variables.
\end{theorem}

\begin{proof}
We spell out the real form used here. Regard $Q$ as a polynomial
$Q_{\mathbb C}$ over $\mathbb C$. The identity
\[
        \det D^2Q\equiv0
\]
remains true after complexification. By the ternary case of the
Gordan--Noether theorem, in the precise form of
\cite[Theorem 5.3]{WatanabeDeBondt2017}, a variable can be eliminated from
$Q_{\mathbb C}$ by a complex linear change of variables. Thus there is
$B\in GL(3,\mathbb C)$ such that
\[
        \widetilde Q(Y)=Q_{\mathbb C}(BY)
\]
is independent of $Y_3$. Differentiating in $Y_3$ gives
\[
        (Be_3)\cdot DQ_{\mathbb C}(BY)=0.
\]
Since $B$ is invertible, this is equivalent to the polynomial identity
\[
        a\cdot DQ_{\mathbb C}(X)=0,
        \qquad
        a=Be_3\neq0 .
\]
Write
\[
        a=b+ic,
        \qquad b,c\in\mathbb R^3.
\]
Because $Q$ has real coefficients, the real and imaginary parts give
\[
        b\cdot DQ=0,
        \qquad
        c\cdot DQ=0.
\]
At least one of $b,c$ is non-zero; call it $v$. Choose a real invertible
matrix $T$ whose third column is $v$, and set
\[
        \widehat Q(Y)=Q(TY).
\]
Then
\[
        \partial_{Y_3}\widehat Q(Y)
        =
        v\cdot DQ(TY)
        =
        0.
\]
Hence $\widehat Q$ is independent of $Y_3$, which is the asserted real
linear reduction. The classical source is Gordan--Noether
\cite{GordanNoether1876}; for the projective geometric formulation of the
same Hesse claim in low dimension, see also \cite{GarbagnatiRepetto2008}.
\end{proof}

\begin{lemma}\label{lem:sheetorder}
Let $m\geq 2$ and let $F=(F_1,F_2,F_3)$ be continuous. Write
\[
        \zeta=\xi+i\eta,\qquad t=x_3,\qquad
        P=F_1+iF_2,\qquad Z=F_3 .
\]
Assume that $P$ is $C^1$ in a punctured neighbourhood of $0$. Suppose that,
after invertible linear changes in the domain and target,
\begin{equation}\label{eq:coneleading}
        P(\zeta,t)=c\zeta^m+R(\zeta,t),
        \qquad c\neq0,
\end{equation}
where
\begin{align}\label{estiR}
     |R(x)|=O(|x|^{m+1}),
        \qquad
        |DR(x)|=O(|x|^m).
\end{align}
  
Then such an $F$ cannot be one-to-one in any neighbourhood of $0$.
\end{lemma}

\begin{proof}
Multiplying the first two target coordinates by a non-zero complex constant
and rotating the $(\xi,\eta)$-plane, we may assume $c=1$. Arguing by
contradiction, assume that $F$ is one-to-one in a ball about the origin.
Choose
\[
        0<\varepsilon<\frac{\pi}{4m}.
\]
For $r>0$, put
\[
        \xi+i\eta=rs e^{i\theta},
        \qquad
        t=r\tau,
\]
and consider
\[
        1-\varepsilon\leq s\leq1+\varepsilon,
        \qquad
        |\tau|\leq\varepsilon .
\]
After rescaling,
\[
        r^{-m}P(rs e^{i\theta},r\tau)
        =
        s^m e^{im\theta}+E_r(s,\theta,\tau),
\]
where
\[
        \|E_r\|_{C^1}\longrightarrow0
\]
on the compact set above as $r\downarrow0$ because of the assumption \eqref{estiR}.

We now prove the monodromy statement on the universal cover of the angular
variable. For $j=0,\ldots,m-1$ define
\[
        \theta=\frac{\varphi+2\pi j}{m}+\alpha
\]
and
\[
        \mathcal G_{r,j}(s,\alpha,\varphi,\tau)
        =
        e^{-i\varphi}r^{-m}
        P\left(
        rs e^{i((\varphi+2\pi j)/m+\alpha)},r\tau
        \right)-1 
\]
for
\[
        |s-1|\leq\varepsilon,\qquad
        |\alpha|\leq\varepsilon,\qquad
        0\leq\varphi\leq2\pi,\qquad
        |\tau|\leq\varepsilon .
\]

We now make the preceding use of the implicit function theorem uniform.  Put
\[
        K=\{(s,\alpha): |s-1|\leq \varepsilon,\ |\alpha|\leq \varepsilon\}
\]
and
\[
        g(s,\alpha)=s^m e^{im\alpha}-1,
        \qquad g:K\subset \mathbb R^2\longrightarrow \mathbb C\simeq\mathbb R^2 .
\]
Since \(0<\varepsilon<\frac{\pi}{4m}\) and \(s>0\) on \(K\), the equation
\[
        g(s,\alpha)=0
\]
has the unique solution \((s,\alpha)=(1,0)\) in \(K\).  Moreover
\[
        D_{(s,\alpha)}g(1,0)(\dot s,\dot\alpha)
        =m\dot s+im\dot\alpha
\]
is an isomorphism from \(\mathbb R^2\) to \(\mathbb R^2\).  Choose
\(\delta>0\) small enough, such that
\[
        \overline{B_\delta(1,0)}\subset \operatorname{int}K
\]
and
\[
        \det D_{(s,\alpha)}g>0
\]
on \(\overline{B_\delta(1,0)}\).  Since \(g\) has no zero on
\(K\setminus B_\delta(1,0)\), compactness of $\Omega$ gives
\[
        \mu:=\min_{K\setminus B_\delta(1,0)} |g|>0 .
\]

By the uniform \(C^1\)-convergence
\[
        G_{r,j}(s,\alpha,\phi,\tau)
        =g(s,\alpha)+o_{C^1}(1),
\]
uniformly for
\[
        (s,\alpha)\in K,\qquad 0\leq \phi\leq 2\pi,
        \qquad |\tau|\leq \varepsilon,
        \qquad j=0,\ldots,m-1,
\]
we may take \(r>0\) sufficiently small so that
\[
        |G_{r,j}-g|<\frac{\mu}{2}
\]
on \(K\setminus B_\delta(1,0)\), uniformly in \((\phi,\tau,j)\), and also
\[
        \det D_{(s,\alpha)}G_{r,j}>0
\]
on \(\overline{B_\delta(1,0)}\), uniformly in \((\phi,\tau,j)\).  Hence
\(G_{r,j}\) has no zero in \(K\setminus B_\delta(1,0)\).

It remains to see that there is exactly one zero inside \(B_\delta(1,0)\).
On \(\partial B_\delta(1,0)\), the homotopy
\[
        H_\lambda=(1-\lambda)g+\lambda G_{r,j},
        \qquad 0\leq \lambda\leq 1,
\]
does not vanish, again for \(r\) sufficiently small.  Therefore the Brouwer
degree is preserved:
\[
        \deg\bigl(G_{r,j}(\cdot,\cdot,\phi,\tau),
        B_\delta(1,0),0\bigr)
        =
        \deg\bigl(g,B_\delta(1,0),0\bigr)
        =1 .
\]
Since \(\det D_{(s,\alpha)}G_{r,j}>0\) in \(B_\delta(1,0)\), every zero of
\(G_{r,j}\) in this ball contributes local degree \(+1\).  Consequently
there is exactly one zero in \(B_\delta(1,0)\), and hence exactly one zero
in \(K\).

We denote this unique zero by
\[
        (s_j(\phi,\tau),\alpha_j(\phi,\tau)).
\]
At this zero the derivative \(D_{(s,\alpha)}G_{r,j}\) is invertible, so the
implicit function theorem with parameters implies that \(s_j\) and
\(\alpha_j\) depend \(C^1\)-smoothly on \((\phi,\tau)\).  Thus, for all
sufficiently small \(r\), for each \(j=0,\ldots,m-1\) and each
\((\phi,\tau)\in[0,2\pi]\times[-\varepsilon,\varepsilon]\), there is a
unique solution
\[
        (s_j(\phi,\tau),\alpha_j(\phi,\tau))
\]
with
\[
        |s_j(\phi,\tau)-1|<\varepsilon,
        \qquad
        |\alpha_j(\phi,\tau)|<\varepsilon .
\]
Setting
\[
        \theta_j(\phi,\tau)
        =
        \frac{\phi+2\pi j}{m}
        +\alpha_j(\phi,\tau),
\]
we obtain
\[
        P\bigl(r s_j(\phi,\tau)e^{i\theta_j(\phi,\tau)},r\tau\bigr)
        =
        r^m e^{i\phi}.
\]

The
functions $s_j,\alpha_j$ are $C^1$ in $(\varphi,\tau)$, and
\[
        \theta_j(\varphi,\tau)
        =
        \frac{\varphi+2\pi j}{m}+\alpha_j(\varphi,\tau)
\]
solves
\begin{equation}\label{eq:sheet-equation}
        P(rs e^{i\theta},r\tau)=r^m e^{i\varphi}
\end{equation}
in the $j$th lifted angular sector.

This construction gives the desired monodromy exactly. Indeed, for
$j=0,\ldots,m-1$, with indices taken modulo $m$,
\[
        \mathcal G_{r,j}(s,\alpha,2\pi,\tau)
        =
        \mathcal G_{r,j+1}(s,\alpha,0,\tau).
\]
By the uniqueness just proved,
\[
        s_j(2\pi,\tau)=s_{j+1}(0,\tau),
        \qquad
        \alpha_j(2\pi,\tau)=\alpha_{j+1}(0,\tau).
\]
Consequently the $j$th sheet at $\varphi=2\pi$ is the $(j+1)$st sheet at
$\varphi=0$, modulo the harmless $2\pi$ ambiguity of the angular variable.

For fixed \(\phi\), put
\[
        x_j^\phi(\tau)
        =
        \bigl(
        r s_j(\phi,\tau)\cos \theta_j(\phi,\tau),
        r s_j(\phi,\tau)\sin \theta_j(\phi,\tau),
        r\tau
        \bigr),
        \qquad -\varepsilon\leq \tau\leq \varepsilon .
\]
Thus, by \eqref{eq:sheet-equation},
\[
        P(x_j^\phi(\tau))=r^m e^{i\phi}
\]
for every \(j=0,\ldots,m-1\) and every
\(\tau\in[-\varepsilon,\varepsilon]\).  Define
\[
        Z_j^\phi(\tau)=Z(x_j^\phi(\tau)),
        \qquad -\varepsilon\leq \tau\leq \varepsilon .
\]

If \(F\) is one-to-one, then each \(Z_j^\phi\) is one-to-one.  Indeed, if
\(\tau_1,\tau_2\in[-\varepsilon,\varepsilon]\) and
\[
        Z_j^\phi(\tau_1)=Z_j^\phi(\tau_2),
\]
then
\[
        P(x_j^\phi(\tau_1))
        =
        r^m e^{i\phi}
        =
        P(x_j^\phi(\tau_2)).
\]
Since \(F=(P,Z)\), we obtain
\[
        F(x_j^\phi(\tau_1))=F(x_j^\phi(\tau_2)).
\]
The injectivity of \(F\) gives
\[
        x_j^\phi(\tau_1)=x_j^\phi(\tau_2).
\]
Comparing the third coordinates, and using \(r>0\), gives
\[
        \tau_1=\tau_2.
\]
Hence \(Z_j^\phi\) is one-to-one.

Set
\[
        I_j(\phi)=Z_j^\phi([-\varepsilon,\varepsilon]).
\]
Since \(Z_j^\phi\) is continuous and \([-\varepsilon,\varepsilon]\) is compact
and connected, \(I_j(\phi)\) is a compact interval in \(\mathbb R\).

We next show that these intervals are pairwise disjoint.  Suppose, to the
contrary, that for some \(j\neq k\),
\[
        I_j(\phi)\cap I_k(\phi)\neq\varnothing .
\]
Then there exist \(\tau,\sigma\in[-\varepsilon,\varepsilon]\) such that
\[
        Z_j^\phi(\tau)=Z_k^\phi(\sigma).
\]
On the other hand, by (2.2),
\[
        P(x_j^\phi(\tau))
        =
        r^m e^{i\phi}
        =
        P(x_k^\phi(\sigma)).
\]
Therefore
\[
        F(x_j^\phi(\tau))=F(x_k^\phi(\sigma)).
\]
Since \(F\) is one-to-one, we have
\[
        x_j^\phi(\tau)=x_k^\phi(\sigma).
\]
Comparing the third coordinates gives
\[
        \tau=\sigma.
\]
Comparing the first two coordinates then gives
\[
        s_j(\phi,\tau)e^{i\theta_j(\phi,\tau)}
        =
        s_k(\phi,\tau)e^{i\theta_k(\phi,\tau)}.
\]
Since \(s_j\) and \(s_k\) are positive, this implies
\[
        \theta_j(\phi,\tau)\equiv \theta_k(\phi,\tau)
        \pmod {2\pi}.
\]
But
\[
        \theta_j(\phi,\tau)-\theta_k(\phi,\tau)
        =
        \frac{2\pi(j-k)}{m}
        +
        \alpha_j(\phi,\tau)-\alpha_k(\phi,\tau).
\]
Since
\[
        |\alpha_j(\phi,\tau)|<\varepsilon,
        \qquad
        |\alpha_k(\phi,\tau)|<\varepsilon,
\]
we have
\[
        |\alpha_j(\phi,\tau)-\alpha_k(\phi,\tau)|<2\varepsilon.
\]
By the choice \(0<\varepsilon<\pi/(4m)\),
\[
        2\varepsilon<\frac{\pi}{2m}.
\]
On the other hand, for \(j\neq k\) with \(0\leq j,k\leq m-1\),
\[
        \operatorname{dist}_{\mathbb R/2\pi\mathbb Z}
        \left(
        \frac{2\pi(j-k)}{m},0
        \right)
        \geq \frac{2\pi}{m}.
\]
Therefore
\[
        \theta_j(\phi,\tau)\not\equiv \theta_k(\phi,\tau)
        \pmod {2\pi},
\]
which is a contradiction.  Hence
\[
        I_j(\phi)\cap I_k(\phi)=\varnothing,
        \qquad j\neq k .
\]

Finally, we record the order consequence.  Write
\[
        a_j(\phi)=\min I_j(\phi),
        \qquad
        b_j(\phi)=\max I_j(\phi).
\]
Since \(Z_j^\phi\) is continuous and one-to-one on
\([-\varepsilon,\varepsilon]\), we have
\[
        I_j(\phi)=[a_j(\phi),b_j(\phi)].
\]
Moreover,
\[
        a_j(\phi)
        =
        \min\{Z_j^\phi(-\varepsilon),Z_j^\phi(\varepsilon)\},
        \qquad
        b_j(\phi)
        =
        \max\{Z_j^\phi(-\varepsilon),Z_j^\phi(\varepsilon)\}.
\]
Thus \(a_j\) and \(b_j\) depend continuously on \(\phi\).  For \(j\neq k\),
the disjointness of \(I_j(\phi)\) and \(I_k(\phi)\) means that exactly one of
the two strict inequalities
\[
        b_j(\phi)<a_k(\phi),
        \qquad
        b_k(\phi)<a_j(\phi)
\]
holds.  Since \(a_j,a_k,b_j,b_k\) are continuous functions of \(\phi\), such a
strict inequality persists for \(\phi\) in a small neighbourhood of the given
parameter value.  Hence the left-to-right order of the labelled intervals
\[
        I_0(\phi),\ldots,I_{m-1}(\phi)
\]
is locally constant in \(\phi\).  Since \([0,2\pi]\) is connected, this order
is the same at \(\phi=0\) and at \(\phi=2\pi\).

On the other hand, the cyclic monodromy of the sheets gives
\[
        I_j(2\pi)=I_{j+1}(0),
        \qquad j=0,\ldots,m-1,
\]
again modulo $m$. A non-trivial cycle of $m\geq2$ objects cannot preserve a
linear order. This contradiction proves the lemma.
\end{proof}

\section{The rank one case}

Throughout the proof we work in a connected coordinate ball. Under
\eqref{eq:noncritical}, equation \eqref{eq:plaplace} is equivalent to
\eqref{eq:nondiv}. The coefficient matrix
\[
        A(q)=I+(p-2)\frac{q\otimes q}{|q|^2}
\]
has eigenvalues
\[
        1,\qquad 1,\qquad p-1.
\]
Thus it is positive definite for every $1<p<\infty$.

\begin{proposition}\label{prop:rankone}
Let $x_0\in\Omega$. If $u$ satisfies \eqref{eq:nondiv}, then
\[
        \rank D^2u(x_0)\neq 1.
\]
\end{proposition}

\begin{proof}
Assume that $D^2u(x_0)$ has rank one. Then
\[
        D^2u(x_0)=\lambda \xi\otimes\xi
\]
with $\lambda\neq 0$ and $\xi\neq0$. Evaluating \eqref{eq:nondiv} at $x_0$
gives
\[
        0
        =
        \lambda\, \xi^TA(Du(x_0))\xi.
\]
Since $A(Du(x_0))$ is positive definite, the last factor is positive. Hence
$\lambda=0$, a contradiction.
\end{proof}

\section{The rank two case}

In this section we exclude points at which $D^2u$ has rank two. The argument
is local, so assume that
\[
        x_0=0,
        \qquad
        \rank D^2u(0)=2.
\]
Choose coordinates so that the $x_3$-axis is the kernel of $D^2u(0)$, and
write
\[
        x'=(x_1,x_2).
\]
Then the $2\times2$ block $D^2_{x'x'}u(0)$ is invertible. By the implicit
function theorem the map
\begin{align}\label{diff2}
            (x',x_3)\longmapsto (D_{x'}u(x),x_3)
\end{align}
is a local diffeomorphism. We denote that
\[
        y=D_{x'}u(x)\in \mathbb{R}^2,\qquad s=x_3,
\]
and in the following, we will calculate under the coordinate $(y,s)$.

As in Rios-Sawyer-Wheeden\cite{AdvMath05}, we define the partial Legendre transform
\[
        \phi(y,s)=x'\cdot y-u(x',s).
\]
Here, and throughout this section, $x'=x'(y,s)$ denotes the inverse branch
given by \eqref{diff2}. The basic identities are
\[
D_y \phi = D_y x' \, y + x' - D_{x'} u \, D_y x'=x',
\]
\[
\phi_s = x_s' \cdot y - D_{x'} u  \cdot x_s' - u_s=-u_s.
\]
then 
\[
        Du=(y,-\phi_s).
\]
Set
\[
B=D^2_{x'x'}u, \qquad P=D_y^2\phi
\]
then 
\[
B=D_{x'}y=(D_yx')^{-1}=(D_y^2\phi)^{-1}=P^{-1}
\]
Since $y=D_{x'}u(x',s)$, differentiation for $s$ gives
\[
        0=D^2_{x'x'}u\,x'_s+D_{x'}u_s.
\]
Thus
\[
        D_{x'}u_s=-Bx'_s.
\]
and
\[
        D_y\phi_s=x'_s=-D_y^2\phi D_{x'}u_s.
\]

Finally,
\[
       \phi_{ss}=-u_{ss}-D_{x'}u_s\cdot x'_s
        =
        -u_{ss}+(x'_s)^TBx'_s,
\]

These identities give the Hessian block factorization
below:
\begin{equation}\label{eq:hessianblocks}
        D^2u=
        \begin{pmatrix}
        B&-Bx'_s\\
        -(x'_s)^TB&(x'_s)^TBx'_s- \phi_{ss}
        \end{pmatrix}.
\end{equation}
\eqref{eq:nondiv} is
\[
        \Delta u+(p-2)\frac{D^2u(Du,Du)}{|Du|^2}=0,
\]

Let
\[
        R=|Du|^2=|y|^2+\phi_s^2,
\]
\[
        A=1+(p-2)\frac{\phi_s^2}{R},
\]
and
\[
        M=I+x'_s\otimes x'_s
        +(p-2)\frac{(y+\phi_sx'_s)\otimes(y+\phi_sx'_s)}{R}.
\]
From \eqref{eq:hessianblocks} one has
\[
        \Delta u=\tr B+(x'_s)^TBx'_s- \phi_{ss},
\]
and, since $Du=(y,-\phi_s)$,
\[
        D^2u(Du,Du)
        =
        (y+\phi_sx'_s)^TB(y+\phi_sx'_s)-\phi_s^2\phi_{ss}.
\]
and $|Du|^2=R$. Therefore
\[
        \tr B+(x'_s)^TBx'_s-\phi_{ss}
        +(p-2)\frac{(y+\phi_sx'_s)^TB(y+\phi_sx'_s)-\phi_s^2\phi_{ss}}{R}=0.
\]
Now let 
\[
w=\phi_{ss}
\]
Substituting these two identities into \eqref{eq:nondiv} gives the
transformed equation
\begin{equation}\label{eq:transformed}
        \tr(BM)=Aw.
\end{equation}

\begin{lemma}\label{lem:AMpositive}
For every $1<p<\infty$,
\[
        A>0,
        \qquad
        M>0.
\]
The positivity is locally uniform as long as $|Du|>0$ and the partial
Legendre transform is defined.
\end{lemma}

\begin{proof}
The assertion $A>0$ follows from
\[
        A=\frac{|y|^2+(p-1)\phi_s^2}{|y|^2+\phi_s^2}.
\]
For $M$, fix $\xi\in\mathbb R^2$ and write
\[
        v=y+\phi_sx'_s.
\]
Then
\[
        \xi\cdot v
        =
        (y,\phi_s)\cdot(\xi,\xi\cdot x'_s),
\]
and Cauchy's inequality in $\mathbb R^3$ gives
\begin{align}\label{cauchy1}
       (\xi\cdot v)^2
        \leq
        (|y|^2+\phi_s^2)\left(|\xi|^2+(\xi\cdot x'_s)^2\right).
\end{align}

If $p\geq2$, positivity is immediate. If $1<p<2$, \eqref{cauchy1}
gives
\[
        \xi^TM\xi
        \geq
        (p-1)\left(|\xi|^2+(\xi\cdot x'_s)^2\right).
\]
This proves the claim.
The local uniformity follows from $R=|Du|^2$ being bounded away from zero
and from the boundedness of $h$ in a sufficiently small coordinate
neighbourhood.
\end{proof}

The local homeomorphism assumption is now applied to the simple map
\[
        G(y,s)=(y,-\phi_s(y,s)).
\]
Let
\[
        H(x)=(D_{x'}u(x),x_3).
\]
After shrinking the neighbourhood, $H$ is a diffeomorphism onto its image and
\[
        Du=G\circ H.
\]
because
\[
\phi_s =-u_s.
\]
Hence
\[
        G=Du\circ H^{-1}
\]
is a local homeomorphism. Furthermore
\[
        DG(y,s)=
        \begin{pmatrix}
        1&0&0\\
        0&1&0\\
        -\phi_{sy_1}& -\phi_{sy_2}& -\phi_{ss}
        \end{pmatrix},
        \qquad
        \det DG=-w.
\]
By Lemma \ref{lem:sign}, there is a number $\sigma\in\{-1,1\}$ such that
\begin{equation}\label{eq:wpositive}
        \sigma w\geq0,
        \qquad
        w(0)=0.
\end{equation}
Here $w(0)=0$ follows from the choice of the $x_3$-axis as the kernel of
$D^2u(0)$. Indeed, \eqref{eq:hessianblocks} gives
\[
        0=D^2u(0)e_3=
        \binom{-B(0)x'_s(0)}{x'_s(0)^TB(0)x'_s(0)-\phi_{ss}(0)}.
\]
Since $B(0)$ is invertible, $D^2u(0)e_3=0$ implies
\[
       x'_s(0)=0,\qquad \phi_{ss}(0)=0.
\]
In particular, \eqref{eq:wpositive} implies
\[
        Dw(0)=0.
\]

\begin{lemma}\label{lem:ranktwoineq}
Under the rank two assumptions above, after \eqref{eq:wpositive} has been
established and after shrinking the coordinate neighbourhood if necessary,
$w$ satisfies a uniformly elliptic differential inequality
\begin{equation}\label{eq:Lineq}
        |\mathcal Lw|\leq C\left(|w|+|Dw|\right)
\end{equation}
in a neighbourhood of $0$, where
\[
        \mathcal Lw
        =
        \tr(BD_y^2w\,BM)
        -2\Gamma\cdot D_yw_s
        +Aw_{ss}
\]
and
\[
        \Gamma
        =
        Bx'_s+(p-2)\frac{\phi_s}{R}B(y+\phi_sx'_s).
\]
\end{lemma}

\begin{proof}
The components of $A$ and $M$ are all smooth in the neighborhood of $x_0$, so they and their derivatives of any order are all bounded.
The variable $y$ is fixed in this differentiation, and $A$ depends only on $y$ and $\phi_s$.,then
\begin{align}\label{As}
      A_s=A_{\phi_s} w,
        \qquad
        A_{ss}=A_{\phi_s} w_s+A_{\phi_s\phi_s}w^2.
\end{align}
Since $M$ is a smooth function of
\[
        y,\qquad \phi_s,\qquad  x'_s,
\]
and the variable $y$ is fixed when differentiating in $s$, then
\begin{align}\label{Ms}
     M_s=M_{\phi_s} w+M_{x^{\alpha}_s}w_\alpha,
\end{align}
Then
\[
        M_s=O(|w|+|Dw|)
\]
Since
\[
        (\phi_s)_s=w,
        \qquad
        (x'_s)_s=D_yw,
\]
Differentiating \eqref{Ms} once more gives
\begin{align}\label{Mss}
        M_{ss}
        ={}&
        M_{\phi_s}w_s+M_{x^{\alpha}_s}w_{\alpha s}
        +M_{\phi_s\phi_s}w^2
        +2M_{\phi_sx^{\alpha}_s}ww_\alpha  \\
        &\quad
        +M_{x^{\alpha}_s x^{\beta}_s}w_\alpha w_\beta .
\end{align}
here the index $\alpha$ and $\beta$ sum from 1 to 2.
By \eqref{eq:wpositive}, the function $w$ has a local extremum at $0$, and hence
$Dw(0)=0$. After shrinking the neighbourhood we may assume
\[
        |w|+|Dw|\leq1.
\]
Since $B=P^{-1}$ and $P_{ss}=(D^2_y\phi)_{ss}=D^2_yw$,
\[
        B_s=-BP_sB,
\]
and
\[
        B_{ss}=2BP_sBP_sB-BD_y^2wB.
\]
Differentiate \eqref{eq:transformed} once in $s$. 
\begin{align}\label{diff1}
     -\tr(BP_sBM)+\tr(BM_s)-A_sw-Aw_s=0.
\end{align}

All coefficients in \eqref{diff1} are bounded in the neighbourhood of $0$,
because $R=|Du|^2$ is bounded away from zero. Hence
\begin{equation}\label{eq:firstder}
        \tr(BP_sBM)=O(|w|+|Dw|).
\end{equation}
The original equation \eqref{eq:transformed} also gives
\begin{equation}\label{eq:traceBM}
        \tr(BM)=O(|w|).
\end{equation}

Differentiating \eqref{eq:transformed} twice in $s$ gives
\[
        0=
        \frac{d^2}{ds^2}\{\tr(BM)-Aw\}
        =
        \tr(B_{ss}M)+2\tr(B_sM_s)+\tr(BM_{ss})
        -A_{ss}w-2A_sw_s-Aw_{ss}.
\]
Using the identities for $B_s$ and $B_{ss}$ above, this becomes
\begin{align}
0={}&
        2\tr(BP_sBP_sBM)
        -\tr(BD_y^2wBM)
        -2\tr(BP_sBM_s)
        +\tr(BM_{ss})
        \notag\\
&       -A_{ss}w-2A_sw_s-Aw_{ss}.
\label{eq:secondder}
\end{align}
Recall
\[
        A=1+(p-2)\frac{\phi_s^2}{|y|^2+\phi_s^2},
\]
since $R=|y|^2+\phi_s^2=|Du|^2$ is bounded away from zero, the derivatives
\[
        A_{\phi_s},\qquad A_{\phi_s\phi_s}
\]
are bounded in the chosen neighbourhood.
By \eqref{As}, all first and second derivatives of $A$ involve no second derivatives of $w$, that is
\[
        A_{ss}w+2A_sw_s=O(|w|+|Dw|).
\]
And $B,P_s$ are bounded because $u$ and $\phi$ are all smooth, so
\[
        \tr(BP_sBM_s)=O(|w|+|Dw|).
\]
It remains to control $2\tr(BP_sBP_sBM)$. For $M$ is positive definite, set
\[
        S=M^{1/2}BM^{1/2},
        \qquad
        N=M^{-1/2}P_sM^{-1/2},
\]
then we have
\[
        \tr(BM)=\tr S,
        \qquad
        \tr(BP_sBM)=\tr(SNS),
\]
and
\[
        \tr(BP_sBP_sBM)=\tr(SNSNS).
\]
The matrix $S$ is symmetric. At the base point $w(0)=0$, and the block
factorization \eqref{eq:hessianblocks} gives
\[
        D^2u(0)=
        \begin{pmatrix}
        B(0)&0\\
        0&0
        \end{pmatrix}.
\]
The two non-zero eigenvalues of $D^2u(0)$ have opposite signs. Indeed, if
they had the same sign, then $D^2u(0)$ would be non-zero semidefinite, and the
elliptic equation
\[
        \tr(a(Du(0))D^2u(0))=0
\]
with $a(Du(0))>0$ would be impossible. Thus $B(0)$ has one positive and one
negative eigenvalue. Because $M(0)>0$, $S(0)=M(0)^{1/2}B(0)M(0)^{1/2}$
has the same inertia as $B(0)$. Hence, after shrinking the
neighbourhood, the two eigenvalues $\lambda,\mu$ of $S$ are opposite in sign
and bounded away from zero.
Choose an orthonormal basis diagonalizing $S$ and write
\[
        S=
        \begin{pmatrix}
        \lambda&0\\
        0&\mu
        \end{pmatrix},
        \qquad
        N=
        \begin{pmatrix}
        \alpha&\beta\\
        \beta&\gamma
        \end{pmatrix}.
\]
Equations \eqref{eq:traceBM} and \eqref{eq:firstder} become
\[
        \lambda+\mu=O(|w|),
\]
and
\[
        \lambda^2\alpha+\mu^2\gamma=O(|w|+|Dw|).
\]
Since $\mu=-\lambda+O(|w|)$ and $|\lambda|$ is bounded away from zero, it
follows that
\[
        \alpha+\gamma=O(|w|+|Dw|).
\]
A direct computation gives
\[
        \tr(SNSNS)
        =
        \lambda^3\alpha^2+\mu^3\gamma^2
        +\lambda\mu(\lambda+\mu)\beta^2.
\]
Using $\mu=-\lambda+O(|w|)$, the preceding estimate
$\alpha+\gamma=O(|w|+|Dw|)$, and the boundedness of $N$, we get
\[
        \tr(BP_sBP_sBM)=\tr(SNSNS)=O(|w|+|Dw|).
\]
The only second derivatives of $w$ contained in $M_{ss}$ are the mixed
derivatives $w_{\alpha s}$. Indeed, $M$ is a smooth function of
\[
        y,\qquad \phi_s,\qquad  x'_s,
\]
and the variable $y$ is fixed when differentiating in $s$. Since
\[
        (\phi_s)_s=w,
        \qquad
        (x'_s)_s=D_yw,
\]
Differentiating \eqref{Ms} once more gives
\begin{align*}
        M_{ss}
        ={}&
        M_{\phi_s}w_s+M_{x^{\alpha}_s}w_{\alpha s}
        +M_{\phi_s\phi_s}w^2
        +2M_{\phi_sx^{\alpha}_s}ww_\alpha  \\
        &\quad
        +M_{x^{\alpha}_s x^{\beta}_s}w_\alpha w_\beta .
\end{align*}
All derivatives of $M$ appearing here are bounded in the chosen
neighbourhood, because $R=|Du|^2$ is bounded away from zero. Hence
\[
        M_{ss}
        =
  M_{x^{\alpha}_s}w_{\alpha s}
        +O(|w|+|Dw|).
\]
Moreover,
\[
        M_{x^{\alpha}_s}
        =
        e_\alpha\otimes x'_s+x'_s\otimes e_\alpha
        +(p-2)\frac{\phi_s}{R}\left(e_\alpha\otimes(y+\phi_sx'_s)
        +(y+\phi_sx'_s)\otimes e_\alpha\right).
\]
Using the symmetry of $B$,
\[
        \tr\bigl(B(e_\alpha\otimes x'_s)\bigr)
        =
        x'_s\cdot Be_\alpha
        =
        (Bx'_s)_\alpha,
\]
and
\[
        \tr\bigl(B(x'_s\otimes e_\alpha)\bigr)
        =
        e_\alpha\cdot Bx'_s
        =
        (Bx'_s)_\alpha.
\]
The two terms containing $y+\phi_sx'_s$ similarly contribute
\[
        2(p-2)\frac{\phi_s}{R}\bigl(B(y+\phi_sx'_s)\bigr)_\alpha.
\]
Therefore
\[
        \tr(BM_{x^{\alpha}_s})=2\Gamma_\alpha.
\]
Consequently
\[
        \tr(BM_{ss})
        =
        2\Gamma\cdot D_yw_s
        +O(|w|+|Dw|).
\]
The first term on the right is kept as part of the principal operator.

Thus \eqref{eq:secondder} implies that 
\begin{equation}\label{eq:Lidentity}
        \mathcal Lw\leq C(|w|+|Dw|).
\end{equation}
Finally we check ellipticity. The principal symbol of $\mathcal L$ is
\[
        \xi^TBMB\xi-2\tau\,\Gamma\cdot\xi+A\tau^2.
\]
Putting
\[
        \eta=B\xi,
        \qquad
        v=y+\phi_sx'_s,
\]
we have
\[
        \xi^TBMB\xi
        =
        \eta^TM\eta
        =
        |\eta|^2+(\eta\cdot x'_s)^2
        +(p-2)\frac{(\eta\cdot v)^2}{R},
\]
\[
        -2\tau\,\Gamma\cdot\xi
        =
        -2\tau\,\eta\cdot x'_s
        -2(p-2)\tau\frac{\phi_s\,\eta\cdot v}{R},
\]
and
\[
        A\tau^2
        =
        \tau^2
        +(p-2)\frac{\phi_s^2\tau^2}{R}.
\]
Therefore the principal symbol equals
\[
        |\eta|^2+(\eta\cdot x'_s-\tau)^2
        +(p-2)\frac{(\eta\cdot v-\tau \phi_s)^2}{R}.
\]
The upper bound
\[
        |\eta|^2+(\eta\cdot x'_s-\tau)^2
        +(p-2)\frac{(\eta\cdot v-\tau \phi_s)^2}{R}\leq C(|\xi|^2+\tau^2).
\]
is trivial, we then prove the lower bound.

For $1<p<2$ we use
\[
        (\eta\cdot v-\tau \phi_s)^2
        \leq
        R\left(|\eta|^2+(\eta\cdot x'_s-\tau)^2\right),
\]
which gives the lower bound
\[
        (p-1)\left(|\eta|^2+(\eta\cdot x'_s-\tau)^2\right).
\]
For the case $p\geq2$, the lower bound $|\eta|^2+(\eta\cdot x'_s-\tau)^2$ is trivial . After shrinking the neighbourhood,
$B^{-1}$ and $x'_s$ are bounded, then
\[
\eta \sim \xi
\]
and \[
|\tau|
\leq
C\left(|\eta|+|\eta\cdot x'_s-\tau|\right)
\]
for
\[
|\tau|
=
|\tau-\eta\cdot x'_s+\eta\cdot x'_s|
\leq
|\eta\cdot x'_s-\tau|+|\eta\cdot x'_s|
\leq
|\eta\cdot x'_s-\tau|+|x'_s||\eta|.
\]
Thus,

\[
        |\eta|^2+(\eta\cdot x'_s-\tau)^2
        \geq c(|\eta\cdot x'_s|^2+\tau^2)\geq c(|\xi|^2+\tau^2).
\]
Thus $\mathcal L$ is uniformly elliptic.
\end{proof}

\begin{proposition}\label{prop:ranktwo}
There is no point $x_0\in\Omega$ such that
\[
        \rank D^2u(x_0)=2.
\]
\end{proposition}

\begin{proof}
Assume such a point exists and use the notation above. By
\eqref{eq:wpositive},
\[
        \sigma w\geq0,
        \qquad
        w(0)=0.
\]
Lemma \ref{lem:ranktwoineq} gives
\[
        |\mathcal Lw|\leq C(|w|+|Dw|).
\]
Set
\[
        v=\sigma w.
\]
Since $-v\leq0$ and $-v$ attains the interior maximum $0$ at the origin, the
strong maximum principle for uniformly elliptic operators
\cite{GilbargTrudinger1983} yields
\[
        v\equiv0
\]
in a neighbourhood of $0$. But then
\[
        w\equiv0,
\]
and therefore
\[
        G(y,s)=(y,-\phi_s(y,s))
\]
is constant in the $s$-direction. Hence, after fixing $y$ and taking two
nearby distinct values of $s$, the map $G$ takes the same value at two
different points. Thus $G$ cannot be locally one-to-one. This contradicts the
local homeomorphism property of $G$.
\end{proof}

\section{The rank zero case}
In this section, we apply the argument similar to that in \cite{GuanWangZhang2016,HanNadirashviliYuan2003} to derive a contradiction.
Let
\[
        \gamma=|Du(0)|>0.
\]
After subtracting the constant \(u(0)\) and rotating coordinates, we may
assume that
\[
        Du(0)=\gamma e_3 .
\]
Since \(D^2u(0)=0\) and \(u\) is analytic, there is a non-zero homogeneous
polynomial \(P_d\) of degree \(d\geq 3\) such that
\[
        u(x)=\gamma x_3+P_d(x)+R(x),
        \qquad R(x)=O(|x|^{d+1}).
\]
We shall use this expansion in the \(C^2\)-sense; in particular,
\[
        DR(x)=O(|x|^d),
        \qquad
        D^2R(x)=O(|x|^{d-1}).
\]

Recall that (1.3) can be written as
\[
        a^{ij}(Du)u_{ij}=0,
        \qquad
        a^{ij}(q)
        =
        \delta_{ij}
        +(p-2)\frac{q_iq_j}{|q|^2}.
\]
Since \(\gamma>0\), the coefficient matrix \(a^{ij}(q)\) is smooth in a
neighbourhood of \(q=\gamma e_3\).  Fix \(y\in S^2\) and put \(x=ry\).
By the homogeneity of \(P_d\),
\[
        Du(ry)
        =
        \gamma e_3+r^{d-1}DP_d(y)+O(r^d),
\]
and
\[
        D^2u(ry)
        =
        r^{d-2}D^2P_d(y)+O(r^{d-1}).
\]
Consequently,
\[
        a^{ij}(Du(ry))
        =
        a^{ij}(\gamma e_3)+O(r^{d-1}).
\]
But
\[
        a^{ij}(\gamma e_3)
        =
        \operatorname{diag}(1,1,p-1).
\]
Evaluating (1.3) at \(ry\), we obtain
\[
\begin{aligned}
0
&=
        a^{ij}(Du(ry))u_{ij}(ry)  \\
&=
        \Bigl(a^{ij}(\gamma e_3)+O(r^{d-1})\Bigr)
        \Bigl(r^{d-2}(P_d)_{ij}(y)+O(r^{d-1})\Bigr)  \\
&=
        r^{d-2}
        \bigl((P_d)_{11}(y)+(P_d)_{22}(y)+(p-1)(P_d)_{33}(y)\bigr)
        +O(r^{d-1}).
\end{aligned}
\]
Dividing by \(r^{d-2}\) and letting \(r\downarrow 0\), we get
\[
        (P_d)_{11}(y)+(P_d)_{22}(y)+(p-1)(P_d)_{33}(y)=0
\]
for every \(y\in S^2\).  Since the left hand side is a homogeneous
polynomial of degree \(d-2\), it vanishes identically.  Hence
\[
        P_{d,11}+P_{d,22}+(p-1)P_{d,33}=0.
\]

Define
\[
        Q(X_1,X_2,X_3)
        =
        P_d(X_1,X_2,\sqrt{p-1}\,X_3).
\]
Then
\begin{equation}\label{eq:Qharmonic}
        \Delta Q=0.
\end{equation}
Writing
\[
        L=\diag(1,1,\sqrt{p-1}),
\]
we have
\[
        D^2Q(X)=L^TD^2P_d(LX)L.
\]
Hence
\[
        \det D^2Q(X)=(p-1)\det D^2P_d(LX).
\]
In particular, the two Hessian determinants have the same sign pattern.

\begin{lemma}\label{lem:rankzeroblowup}
Under the assumptions of Theorem \ref{thm:main},
\[
        \det D^2Q\equiv0.
\]
\end{lemma}

\begin{proof}
By Lemma \ref{lem:sign}, $\det D^2u$ has a fixed weak sign near $0$. If
$\det D^2P_d$ took both positive and negative values on the unit sphere,
then
\[
        \det D^2u(rx)
        =
        r^{3(d-2)}\det D^2P_d(x)
        +
        O(r^{3d-5})
\]
would force $\det D^2u$ to take both signs in every neighbourhood of $0$.
Therefore $\det D^2P_d$, and hence $\det D^2Q$, has a fixed weak sign.

If $\det D^2Q\leq0$, then the trace-free matrix $D^2Q$ has at most one
negative eigenvalue at every point. Theorem \ref{thm:LGW} applies to $Q$.
Since $d\geq3$,
\[
        D^2Q(0)=0,
\]
so $\det D^2Q$ vanishes at $0$. Hence
\[
        \det D^2Q\equiv0.
\]
If $\det D^2Q\geq0$, the same argument applies to $-Q$. This proves the
claim.
\end{proof}

\begin{proposition}\label{prop:rankzero}
There is no point $x_0\in\Omega$ such that
\[
        D^2u(x_0)=0.
\]
\end{proposition}

\begin{proof}
Assume that $D^2u(0)=0$ and use the notation above. By Lemma
\ref{lem:rankzeroblowup},
\[
        \det D^2Q\equiv0.
\]
The ternary zero-Hessian theorem, Theorem \ref{thm:Hesse}, implies that
after an invertible real linear change of variables,
\[
        Q(X)=Q_0(X_1,X_2).
\]
Since $Q$ is harmonic, $Q_0$ solves a constant-coefficient elliptic equation
in two variables. More explicitly, after the real linear change which makes
$Q$ independent of the third variable, the Euclidean Laplacian becomes
\[
        g^{ij}\partial_{ij}
\]
with $(g^{ij})$ positive definite, and the restriction to the first two
variables remains positive definite. A further real linear change in the
$(X_1,X_2)$ variables therefore reduces this equation to the two-dimensional
Laplace equation. Hence, because $Q_0$ is homogeneous,
\[
        Q_0(X_1,X_2)
        =
        \operatorname{Re}\left(c(X_1+iX_2)^d\right)
\]
for some $c\neq0$.

Returning to the original variables only composes the leading gradient
$DP_d$ with invertible linear maps in the domain and target. Therefore the
gradient map
\[
        Du(x)-Du(0)
\]
has, after invertible linear changes, the form
\[
        \left(
        \operatorname{Re} c_1(\xi+i\eta)^{d-1},
        -\operatorname{Im} c_1(\xi+i\eta)^{d-1},
        0
        \right)
        +
        O\left(|(\xi,\eta,t)|^d\right),
\]
with $c_1\neq0$. Here
\[
        d-1\geq2.
\]
Since $u$ is analytic, the last expansion is an expansion in the $C^1$
sense; equivalently, the remainder in its first two components satisfies
the derivative bound required in Lemma \ref{lem:sheetorder}. The sign in
the second component is absorbed by the permitted reflection
$\eta\mapsto-\eta$.
This contradicts Lemma \ref{lem:sheetorder}, which proves the mapping $Du$ cannot be one to one.
\end{proof}

\section{Proof of the main theorem}

\begin{proof}[Proof of Theorem \ref{thm:main}]

Suppose that $\det D^2u(x_0)=0$ for some $x_0\in\Omega$. Then
\[
        \rank D^2u(x_0)\in\{0,1,2\}.
\]
Rank one is excluded by Proposition \ref{prop:rankone}. Rank two is
excluded by Proposition \ref{prop:ranktwo}. Rank zero is excluded by
Proposition \ref{prop:rankzero}. Therefore no such $x_0$ exists, and
\[
        \det D^2u(x)\neq0
        \qquad x\in\Omega.
\]
The inverse function theorem gives that $Du$ is a local diffeomorphism. If
$Du$ is globally one-to-one onto its image, the local inverse branches agree
and $Du$ is a diffeomorphism onto $Du(\Omega)$.
\end{proof}

\small
\bibliographystyle{plain}
\bibliography{reference}

{\small
\noindent
(Jiahuan Li) School of Mathematical Sciences, University of Science and
Technology of China, Hefei, 230026, Anhui Province, P.R. China.\\
Email address: \href{mailto:jiahuan@mail.ustc.edu.cn}{jiahuan@mail.ustc.edu.cn}\par\medskip
\noindent
(Yilu Liu) School of Mathematical Sciences, University of Science and
Technology of China, Hefei, 230026, Anhui Province, P.R. China.\\
Email address: \href{mailto:liuylgeoanaly@mail.ustc.edu.cn}{liuylgeoanaly@mail.ustc.edu.cn}\par\medskip
\noindent
(Xi-Nan Ma) School of Mathematical Sciences, University of Science and
Technology of China, Hefei, 230026, Anhui Province, P.R. China.\\
Email address: \href{mailto:xinan@ustc.edu.cn}{xinan@ustc.edu.cn}
}

\end{document}